\documentclass{article}
\usepackage[utf8]{inputenc}
\usepackage{amssymb,amsfonts,amsmath, amsthm}
\usepackage{graphics,graphicx}
\usepackage[mathlines]{lineno}
\usepackage{bm} 
\usepackage{todonotes} 
\usepackage{tikz}
\usepackage{authblk}
\usetikzlibrary{backgrounds}
\usetikzlibrary{arrows}
\usetikzlibrary{shapes,shapes.geometric,shapes.misc}

\newtheorem{thm}{Theorem}[section]
\newtheorem{prop}[thm]{Proposition}
\newtheorem{lem}[thm]{Lemma}
\newtheorem{cor}[thm]{Corollary}
\newtheorem{rem}[thm]{Remark}
\newtheorem{defn}[thm]{Definition}

\newtheorem{obs}[thm]{Observation}

\theoremstyle{definition}

\pgfdeclarelayer{edgelayer}
\pgfdeclarelayer{nodelayer}
\pgfsetlayers{background,edgelayer,nodelayer,main}
\tikzstyle{none}=[inner sep=0mm]
\tikzstyle{whitevertex}=[fill=white, draw=black, shape=circle]
\tikzstyle{blackvertex}=[fill=black, draw=black, shape=circle]
\tikzstyle{dgreyvertex}=[fill={rgb,255: red,128; green,128; blue,128}, draw=black, shape=circle]
\tikzstyle{lgreyvertex}=[fill={rgb,255: red,191; green,191; blue,191}, draw=black, shape=circle]

\tikzstyle{thick}=[-, thick]
\tikzstyle{light}=[-, draw={rgb,255: red,191; green,191; blue,191}]

\title{Bounds on the propagation radius in power domination}
\author[1,3]{Imran Allie\footnote{imran.allie@uct.ac.za, corresponding author}}
\author[1,3]{Brandon du Preez \footnote{brandon.dupreez@uct.ac.za}}
\author[1,3]{Dean Reagon \footnote{RGNDEA001@myuct.ac.za}}
\author[2,3]{Adriana Roux \footnote{rianaroux@sun.ac.za}}
\affil[1]{Department of Mathematics and Applied Mathematics, University of Cape Town, Rondebosch, Cape Town, 7700, South Africa}
\affil[2]{Department of Mathematical Sciences, Stellenbosch University, Stellenbosch, 7600, South Africa}
\affil[3]{DSI-NRF Centre of Excellence in Mathematical and Statistical Sciences (CoE-MaSS), South Africa}

\begin{document}
	
	\maketitle
	
	\begin{abstract}

		Let $G$ be a graph and let $S \subseteq V(G)$. It is said that $S$ \textit{dominates} $N[S]$. We say that $S$ \textit{monitors} vertices of $G$ as follows. Initially, all dominated vertices are monitored. This step is called the \textit{domination} step. Thereafter, the set of unmonitored vertices of which each is the only unmonitored neighbour of a monitored vertex, is monitored. 
		This step is called a \textit{propagation} step and is repeated until the process terminates. The process terminates when the there are no monitored vertices which have exactly one unmonitored neighbour. This combined process of initial domination and subsequent propagation is called \textit{power domination}. If all vertices of $G$ are monitored at termination, then $S$ is said to be a \textit{power dominating set (PDS) of $G$}. The \textit{power domination number of $G$}, denoted as $\gamma_p(G)$, is the minimum cardinality of a PDS of $G$. The \textit{propagation radius of $G$} is the minimum number of steps it takes a minimum PDS to monitor $V(G)$. In this paper we determine an upper bound on the propagation radius of $G$ with regards to power domination, in terms of $\delta$ and $n$. We show that this bound is only attained when $\gamma_p(G)=1$ and then improve this bound for $\gamma_p(G)\geq 2$. Sharpness examples for these bounds are provided. We also present sharp upper bounds on the propagation radius of split graphs. We present sharpness results for a known lower bound of the propagation radius for all $\Delta\geq 3$.
		
	\end{abstract}

	\section{Introduction}{\label{Introduction}}
	
	\textit{Power domination theory}, which has received some attention in the literature in the past decade, is a derivative of the much studied branch of graph theory known as domination theory and has applications in the monitoring of electrical networks. In electrical networks, \textit{Phaser Measurement Units (PMUs)} are placed at nodes in the network for the purpose of monitoring the remainder of the nodes. For cost reasons, it is desirable to minimize the number of PMUs in the network, whilst it still being the objective to monitor the entire network. Graph theoretically, this translates to the problem of minimizing the \textit{power domination number} of a graph (denoted as $\gamma_p(G)$ where $G$ is the graph), with the power domination number corresponding to the number of PMUs placed at nodes in the network. Various results have been proven in the literature regarding the power domination number of graphs. See for example \cite{pdgg_dorfling_henning_2006, vvh_2018}. 
	
	Of practical concern is not just the number of PMUs placed in an electrical network, but also the time taken for the set of PMUs to completely monitor the network. Letting the network be a graph $G$ and the set of nodes which have PMUs being our power dominating set $S$, the monitoring of other vertices occurs by an initial domination step (as domination is typically defined in graph theory), and then a well-defined step-by-step process called \textit{propagation}. The time taken for the set of PMUs to monitor the network is measured by the number of steps required, as per the well-defined propagation process. The minimum number of steps required for a minimum power dominating set to monitor the graph is called the \textit{propagation radius}, denoted as $rad_p(G)$ (formal definitions to follow). We note that some authors have referred to the propagation radius instead as propagation time. Not much is present in the literature regarding the propagation radius. In \cite{pdbtc_liao_2016}, however, the power domination of graphs are considered where the radius is restricted, and in \cite{prop_time_2012} the propagation time with regards to \textit{zero forcing} is investigated (zero forcing has the same propogation steps as power domination, but do not have an initial domination step.).  
	
	In this paper, we investigate bounds on the propagation radius in graphs. We determine an upper bound on $rad_p(G)$ in terms of $\delta$ and $n$ and present sharpness examples. For this upper bound, sharp example graphs must have $\gamma_p(G) = 1$. We then determine further bounds in terms $\delta$, $n$ and $\gamma_p(G)$ where $\gamma_p(G) \geq 2$, and present sharpness examples for some cases. We also consider \textit{split graphs}, where we determine a sharp upper bound on the propagation radius in terms of $n$ and $\gamma_p(G)$. We present a sharp lower bound of $rad_p(G)$ in terms of $n$, $\Delta$ and $\gamma_p(G)$.

	\section{Preliminaries} 
	
	\begin{defn} \label{def_powerdomination} {\rm
			Let $G$ be a graph and let $S \subseteq V(G)$. We denote by $P^i_G(S)$ the set of vertices \textit{monitored by $S$ in $i$ steps} for $i \geq 1$. We define $P^i_G(S)$ recursively as follows:
			\begin{itemize}
				\item $P^1_G(S) = N[S]$.
				\item For $i \geq 2$, let $V^i = \{v \in P^i_G(S) : |N[v] \setminus P^i_G(S)| \leq 1 \}$. Then $P^{i+1}_G(S) = P^i_G(S) \cup N[V^i]$. 
		\end{itemize}}
	\end{defn}
	
	The step from $S$ to $P^1_G(S) = N[S]$ is called the \textit{domination step}. For $i \geq 2$, the step from $P^{i-1}_G(S)$ to $P^{i}_G(S)$ is called a \textit{propagation step}. If there is no confusion, we will simply use $P^i(S)$ or $P^i$. For $v \in P^i$ such that $v$ has exactly one unmonitored neighbour $u$, we say that $v$ \textit{propagates to} $u$ at step $i+1$. The sequence $P^i_G(S)_{i \geq 1}$ is called a \textit{propagation sequence}. Clearly, $P^i_{i \geq 1}$ is such that $P^i \subseteq P^{i+1}$ and that if $P^i = P^{i+1}$, then $P^i$ is a fixed point of the sequence. If $i$ is the smallest integer such that $P^i = P^{i+1}$ then we denote $P^i$ as $P^{\infty}$. We say that $S$ \textit{power dominates} $P^{\infty}$. $S$ is a \textit{power dominating set of $G$} (or a \textit{PDS of $G$}, for brevity) if $P^{\infty} = V(G)$. The \textit{power domination number of $G$}, denoted $\gamma_p(G)$, is the minimum cardinality of a PDS of $G$. 
	The following definition will also be useful for our purposes.

	\begin{defn} \label{def_live_set}
		For $i>1$, the \textit{live set} of $P^i_G(S)$, denoted $L^i_G(S)$, is the subset of vertices in $P^i_G(S)$ which have not propagated. 
		We define the live set of $P^1_G(S)$ to be $L^i_G(S)=N(S)\setminus S$.
		Since any vertex can propagate to at most one vertex, we have that the sequence $|L^1|, |L^2|, \dots $ is non-increasing.
	\end{defn}
	
	\begin{defn}  \label{def_propagation_radius} {\rm
			Let $G$ be a graph and let $S$ be a PDS of $G$. The \textit{propagation radius of $G$ with respect to $S$}, denoted by $rad_p(G,S)$, is defined as $rad_p(G,S) = \min\{i : P^i_G(S) = V(G)\}$. The \textit{propagation radius of $G$}, denoted by $rad_p(G)$, is defined as $rad_p(G) = \min\{rad_p(G,S) : S~\text{is a PDS of}~G\}$. }		
	\end{defn}
	
	We say that a set $S\subseteq V$ is an \textit{optimal} PDS if $|S| = \gamma_p(G)$ and $rad_p(G,S) = rad_p(G)$.
	
	Let $S$ be a subset of $V(G)$, and suppose $u\in S$. 
	A vertex $v$ is an $S$-\textit{private neighbour} of $u$ if $v$ is adjacent to $u$, but is not adjacent to any other vertex of $S$. 
	Further, $v$ is an \textit{external} $S$-private neighbour if $v\not\in S$.
	
	Zero forcing is a process which is very similar to power domination. It was first defined in \cite{zeroforcing}. For our purposes we make use of some known results on zero forcing. Let $G$ be a graph and let $S \subseteq V(G)$. We term the vertices in $S$ as \textit{monitored} and the vertices in $V(G)\setminus S$ as \textit{unmonitored}. The \textit{forcing rule} is as follows: If $v \in V(G)$ is monitored and has exactly one unmonitored neighbour $u$, then $v$ \textit{forces} $u$, denoted by $v \rightarrow u$, so that $u$ becomes monitored. $S$ is a \textit{zero forcing set of $G$} (or a \textit{ZFS of $G$}, for brevity) if repeated application of the forcing rule on $S$ in $G$ eventually results in all vertices in $G$ being monitored. Clearly, zero forcing is identical to power domination except that the initial domination step of power domination is not present in zero forcing. It follows that if $S$ is a PDS of $G$ then $N[S]$ is a ZFS of $G$. 
	
	For a given ZFS of $G$, say $S$, we may list each single force which occurs chronologically until all of $V(G)$ is monitored. Such a chronological list of forces is called a \textit{forcing sequence of $G$ with respect to $S$}. We note that for a given ZFS of $G$, there may be multiple possible forcing sequences. Let $T$ be a forcing sequence of $G$ with respect to $S$, a ZFS of $G$. A \textit{forcing chain with respect to $T$} is a sequence of vertices $v_0, \dots, v_k$ such that $v_i \rightarrow v_{i+1}$ is listed in $T$ for each $i \in \{0,\dots,k-1\}$. A \textit{maximal forcing chain with respect to $T$} is a forcing chain with respect to $T$ which is not a proper subsequence of some other forcing chain with respect to $T$. We note that any $v \in S$ which does not force any other vertex in $T$, is considered a forcing chain on its own. Since it is not properly contained in any other forcing chain, it is furthermore a maximal forcing chain. A \textit{reversal} of $S$ is a set of terminal vertices of all the maximal forcing chains with respect to some forcing sequence. 
	
	\begin{lem}[\cite{zf_bar_2010}]
		\label{lem:reversal}
		Let $G$ be a graph and let $S$ be a ZFS. Then every reversal $Z$ of $S$ is a ZFS of $G$. Furthermore, $S$ is a reversal of $Z$.
	\end{lem}
	\begin{proof}
		In the proof of Theorem 2.6 of \cite{zf_bar_2010}, it was shown that every reversal $Z$ of a ZFS is also a ZFS. Further, the authors show that $S$ is a reversal of $Z$.
	\end{proof}

	If there is no confusion, then we will simply refer to a forcing chain instead of a forcing chain with respect to $T$. 
	
	\begin{figure} [h!]
		\begin{center}	
			\begin{tikzpicture} [every node/.style={draw,shape=circle,fill=black,text=white,scale=0.8},scale=1]
				\node (0) at (-2, 1.5) [label={[color=black]:$a$}] {};
				\node (1) at (-2, -1.5) [label={[color=black]:$b$}] {};
				\node [fill=none] (2) at (-1, 0.75) [label={[color=black]:$c$}] {};
				\node [fill=none] (3) at (-1, -0.75) [label={[color=black]:$d$}] {};
				\node [fill=none] (4) at (0, 0) [label={[above,color=black]:$e$}] {};
				\node [fill=none] (5) at (1.25, 0) [label={[above,color=black]:$f$}] {};
				\draw (0) -- (2) -- (4) -- (5);
				\draw (1) -- (3) -- (4);
			\end{tikzpicture}
		\end{center}
		\caption{A graph $G$ with $\{a,b\}$ being both a PDS and ZFS.} \label{fig_example}
	\end{figure}
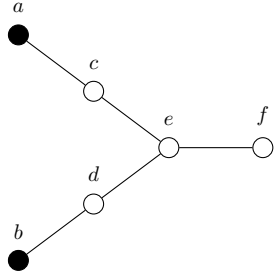
	
	Figure \ref{fig_example} serves to provide clarity on the definitions in this paragraph. In the graph $G$, the set $S = \{a,b\}$ is a PDS. 
	We have that $a$ dominates $c$ and $b$ dominates $d$, then $c$ propagates to $e$ and $d$ propagates to $e$, after which $e$ propagates to $f$. 
	The propagation sequence is $P^1 = \{a,b,c,d\}, P^2 = \{a,b,c,d,e\}, P^3 = P^{\infty} = V(G)$. 
	Note that both $c$ and $d$ are considered to propagate to $e$. The live sets at every step are $L^1 = \{c,d\}$, $L^2 = \{e\}$, and $L^3 = \{f\}$.
	Note that $S$ is also a ZFS. There are two possible forcing sequences with respect to $S$: $T_1 = \{a \rightarrow c, b \rightarrow d, c \rightarrow e, e \rightarrow f\}$ and $T_2 = \{a \rightarrow c, b \rightarrow d, d \rightarrow e, e \rightarrow f\}$. With respect to $T_1$: we have maximal forcing chains $a \rightarrow c \rightarrow e \rightarrow f$ and $b \rightarrow d$; and we have a reversal $\{f,d\}$ which is then also a ZFS. With respect to $T_2$: we have maximal forcing chains $b \rightarrow d \rightarrow e \rightarrow f$ and $a \rightarrow c$; and we have a reversal $\{f,c\}$ which is then also a ZFS.

	\section{Upper bounds when $\bm{\gamma_p = 1}$}
	
	\label{sec:upp_gam_1}
	
	Let $G$ be a graph and let $S$ be a $\gamma_p$-set of $G$. Consider the sequence of propagations in $G$ that begins with $S$. 
	In the first step, the set $N[S]$ is monitored.
	Each propagation step after that monitors at least one new vertex.
	Thus, we make the following observation: If $G$ is a graph with order $n$ and $S$ a $\gamma_p$-set of $G$ that monitors $G$ in $rad_p(G)$ steps, then $n \geq |N[S]| + (rad_p(G) - 1)$. Therefore, $rad_p(G) \leq n - |N[S]|+ 1$. 
	
	\begin{obs} \label{obs}
		Let $G$ be a graph and let $S$ be a $\gamma_p$-set of $G$. Then $rad_p(G) \leq n - |N[S]|+ 1$. 
	\end{obs}
	
	\begin{prop} \label{prop_attainobsbound}
		Let $G$ be a graph with order $n$ and minimum degree $\delta$. Then $rad_p(G) \leq n - \delta$. Moreover, if $G = K_n$ or $G = K_n - M$ where $M$ is a perfect matching, then $rad_p(G) = n-\delta$.
		\begin{proof}
			Let $S$ be a $\gamma_p$-set of $G$. Since $ \delta + 1 \leq |N[S]|$, we have that $rad_p(G) \leq n - (\delta + 1) + 1 = n - \delta$, by Observation \ref{obs}. Let $G = K_n$. Then $\gamma_p(G) = 1$, $\delta = n-1$ and $rad_p(G) = 1 = n - (n-1) = n - \delta$. Let $G = K_n - M$ where M is a perfect matching. In this case $n$ must be even and $\delta = n-2$. Let	$v \in V(G)$ and let $S = \{v\}$. Then $N[v] = V(G)\setminus \{u\}$ where $uv \in M$. There is now only one vertex that is unmonitored, so it is monitored in the next step. 
			Then $S$ is a $\gamma_p$-set of $G$ and $rad_p(G) = 2 = n - (n-2) = n - \delta$.
		\end{proof}
	\end{prop}
	
	\begin{thm}
		\label{thm:delta_1_up_bound}
		Let $G$ be a graph such that $G \neq K_n$ and $G \neq K_n - M$ where $M$ is a perfect matching. Then $rad_p(G) \leq n-\delta(G)-1$. 
		\begin{proof}
			For any $G$ we have that $n \geq \delta + 1$. 
			If $n=\delta + 1$ then $G = K_n$. If $n=\delta + 2$ then either $\Delta = \delta$ or $\Delta = \delta + 1$. If the former, then $G = K_n - M$ where $M$ is a perfect matching. If $\Delta = \delta + 1$, then $\Delta(G)=n-1$ and $rad_p(G)=1$, in which case we have $rad_p(G)=1=n-\delta(G)-1$. For the rest of the proof we may assume that $n \geq \delta + 3$.
			
			Let $S$ be a $\gamma_p$-set of $G$ such that $rad_p(G,S) = rad_p(G)$. We assume to the contrary that $rad_p(G) > n-\delta(G)-1$. By Proposition \ref{prop_attainobsbound}, we know that $rad_p(G) \leq n-\delta$. Therefore, $rad_p(G) = n-\delta$. It follows that $S = \{v\}$ for some $v$ with degree $\delta$. It follows as well that exactly one vertex is monitored at each step except for the initial domination step. That is, $|P^i \setminus  P^{i-1}| = 1$ for each $i \geq 2$. 
			Since $n \geq \delta+3$ and $rad_p(G) = n-\delta$, we also have that $rad_p(G) = n - \delta \geq \delta + 3 - \delta = 3$. 
			Recall as well that $N[v]$ is a ZFS. Let $T = \{p_1 \rightarrow x_1, \dots, p_k \rightarrow x_k\}$ be a forcing sequence with respect to $N[S]$. Note that no two of these forces occur simultaneously, so that these forces coincide with the sequence of propagations. 
			That is $P^i \setminus P^{i-1}=\{x_{i-1}\}$. We have $p_i \in L^i$ for each $i$ and $N(v) = L^1$. 
			
			\emph{Claim 1}: $|L^i| = \delta$ for all $i \in \{1, ..., k\}$ and $L^k=N(x_k)$. 
			
			Since $L^1=N(S)\setminus S$, and $|N(S)|=\delta$, we must have $|L^1|=\delta$. Since $x_k$ is the last vertex to be monitored, we must have that all of its neighbours propagate to it. Then since it has at least $\delta$ neighbours and all of them are in $L^{k}$ we have $|L^{k}| \geq \delta$. Since $|L^i|$ is non-increasing we have that $|L^i| = \delta$ for all $i \in \{1,\dots, k\}$.
			
			\emph{Claim 2}: For any forcing sequence of $N[S]$, the terminal vertex of every maximal forcing chain is in $N[x_k] \cup \{v\}$.
			
			Note that $v$ is its own maximal forcing chain.
			Assume to the contrary that $u \notin N[x_k] \cup \{v\}$ is the terminal vertex of a maximal forcing chain. 
			It follows from Claim 1 that $u$ is not in $L^k$.
			Let $i$ be the largest integer such that $u \in L^i$. There must exist some $w \in L^i$ that propagates to some vertex in $P^{i+1}$ such that $w \neq u$. Then $L^{i+1} \subseteq (L^i  \setminus  \{u,w\}) \cup \{x_{i+1}\}$. But then $|L^{i+1}| < |L^i|$ which contradicts Claim 1.
			
			\emph{Claim 3}: The set $\{x_k\}$ is a PDS of $G$.
			
			Since $\{v\}$ is a power dominating set, we can say that $N_G(v)$ is a ZFS of $G-v$ and that by Claim 2, the terminal vertex of every maximal forcing chain is in $N[x_k]$. Let $Z'$ be a reversal of $N_G(v)$ in $G-v$. Then $Z' \subseteq N[x_k]$. We also know that any reversal of $N_G(v)$ in $G-v$ is a ZFS of $G-v$. Therefore, $Z'$ is a ZFS of $G-v$. 
			By Lemma \ref{lem:reversal} we have that $N_G(v)$ is a reversal of $Z'$ in $G-v$. 
			By definition of a reversal, $N_G(v)$ is the set of the terminal vertices of maximal forcing chains of some forcing sequence $T'$ of $Z'$ in $G-v$. It follows that $Z'$ monitors at least $V(G)\setminus\{v\}$ in $G$. Then, since only $v$ is unmonitored in $G$ by $Z'$, it is monitored in the next step. Since $Z' \subseteq N[x_k]$, $N[x_k]$ is a ZFS of $G$, which implies that $\{x_k\}$ is a PDS of $G$.
			
			\emph{Claim 4}: $x_{k-1}$ is adjacent to $x_k$.
			
			Assume to the contrary that $x_{k-1}$ is not adjacent to $x_k$. Then, at the second last propagation step in which $p_{k-1}$ forces $x_{k-1}$, all the neighbours of $x_k$ are already monitored, implying that $p_{k-1}$ forces $x_{k-1}$ and $p_{k}$ forces $x_{k}$ in the same propagation step, a contradiction. 
			
			By Claim 4, $x_{k-1}$ propagates to $x_k$, so we may choose $p_k$ such that $p_k = x_{k-1}$. If any vertex in $N(x_k) \setminus \{x_{k-1}\}$ is not adjacent to $x_{k-1}$, it will propagate to $x_k$ before the last step, contradicting the propagation radius of G. Therefore $N[x_k] \subseteq N[x_{k-1}]$ and by Claim 3 it follows that $\{x_{k-1}\}$ is a power dominating set of $G$. Recall that $x_k$ is monitored after $x_{k-1}$. Then $p_{k-1}$ cannot be adjacent to $x_k$ otherwise it would be adjacent to two unmonitored vertices, $x_k$ and $x_{k-1}$, when it propagates to $x_{k-1}$. Therefore $|N[x_{k-1}]| \geq |N[x_k]| + |\{p_{k-1}\}| = (\delta + 1) + 1$. Then $rad_p(G, \{x_{k-1}\}) \leq n-|N[x_{k-1}]|+1 \leq n-\delta-2+1 = n-\delta-1$. But this contradicts that $rad_p(G,S) = rad_p(G)$. Therefore $rad_p(G) \neq n-\delta$ implying that $rad_p(G) \leq n - \delta - 1$.
			
		\end{proof}
	\end{thm}
	
	The next lemma shows that if two distinct vertices of $G$ are adjacent to exactly the same vertices (except for the vertices themselves), then at least one of these vertices must belong to the neighbourhood of any PDS of $G$. 
	
	\begin{lem} \label{lem_fort} (Proposition 9.15 \cite{HL_Inverse_zeroforcing})
		Let $G$ be a graph, $S$ be a PDS of $G$, and $u$ and $v$ be vertices in $G$ such that $N(u)  \setminus  \{v\} =
		N(v)  \setminus  \{u\}$. Then, $N[S] \cap \{u, v\} \neq \emptyset$.
		\begin{proof}
			Assume to the contrary that there exists a PDS $S$ such that $N[S] \cap \{u, v\} = \emptyset$. We have that $N[S] \subseteq V(G)  \setminus  \{u,v\}$ so that $V(G)  \setminus  \{u,v\}$ is a ZFS of $G$. Let $u$ be forced before $v$. Let $w \neq v$ force $u$. Then $w \in N(u)  \setminus  \{v\} = N(v)  \setminus  \{u\}$. But this implies that $w$ has two unmonitored neighbours, $u$ and $v$, when forcing $u$, a contradiction.
		\end{proof}
	\end{lem}
	
	\begin{thm} \label{thm_upperboundexample}
		For each positive integer $\delta \geq 2$, there exists infinitely many graphs with minimum degree $\delta$, order $n$, and $rad_p(G) = n - \delta - 1$.
		\begin{proof}
			We divide the proof into cases that depend on the value of $\delta$.
			
			\textit{Case 1:} $\delta = 2$. 
			For $k\geq 3$, we construct a graph $G(2,k)$.
			Let $G(2,k)$ have vertices $u_i$ and $v_i$ for each $1\leq i \leq k$, and three extra vertices $x,y$ and $z$.
			Let $G(2,k)$ have edges: $xu_1, xv_1, xz$ incident with $x$; edges $yu_1, yv_1, yz$ incident with $y$; all possible edges of the form $u_iu_{i+1}$ and $v_iv_{i+1}$, $1 \leq i \leq k-1$; for $i$ an even integer with $2 \leq i \leq k$, all possible edges of the form $u_iv_{i+1}$ and $u_iv_{i-1}$; and $u_kv_k$ (see Figure \ref{fig:g_2_example} for an example where $k=4$).
			The graph $G(2,k)$ has minimum degree 2 and order $n = 2k+3$. 
			The singleton set $\{x\}$ is a PDS as it first dominates its neighbourhood, after which it propagates $z \to y, u_1 \to u_2, v_1 \to v_2, v_2 \to v_3, u_2 \to u_3$ and so on. 
			Similarly $\{y\}$ is a PDS, and both have power domination radius $2k = n-\delta-1$.
			To show that $rad_p(G(2,k)) = 2k$, it suffices to show that no other singleton set is a PDS. 
			Note $N(x) = N(y)$.
			Thus by Lemma \ref{lem_fort}, the only other singletons to consider are $\{z\}, \{u_1\}$ and $\{v_1\}$ By inspection, none of these are PDSs.

			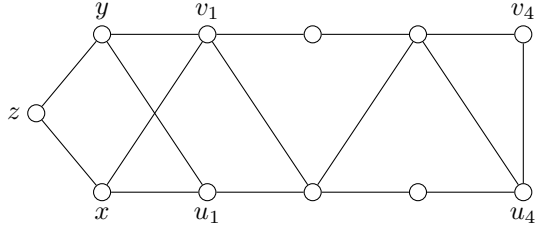
\begin{figure}[h!]
				\centering
				\begin{tikzpicture}[scale=0.7, inner sep=0.8mm] 
					\begin{pgfonlayer}{nodelayer}
						\node [style=whitevertex, label={left:$z$}] (0) at (-4.25, 0) {};
						\node [style=whitevertex, label={above:$y$}] (1) at (-3, 1.5) {};
						\node [style=whitevertex, label={below:$x$}] (2) at (-3, -1.5) {};
						\node [style=whitevertex, label={above:$v_1$}] (3) at (-1, 1.5) {};
						\node [style=whitevertex, label={below:$u_1$}] (4) at (-1, -1.5) {};
						\node [style=whitevertex] (5) at (1, 1.5) {};
						\node [style=whitevertex] (6) at (1, -1.5) {};
						\node [style=whitevertex] (7) at (3, 1.5) {};
						\node [style=whitevertex] (8) at (3, -1.5) {};
						\node [style=whitevertex, label={above:$v_4$}] (9) at (5, 1.5) {};
						\node [style=whitevertex, label={below:$u_4$}] (10) at (5, -1.5) {};
					\end{pgfonlayer}
					\begin{pgfonlayer}{edgelayer}
						\draw (0) to (1);
						\draw (0) to (2);
						\draw (1) to (3);
						\draw (3) to (5);
						\draw (5) to (7);
						\draw (7) to (9);
						\draw (2) to (4);
						\draw (4) to (6);
						\draw (6) to (8);
						\draw (8) to (10);
						\draw (1) to (4);
						\draw (2) to (3);
						\draw (3) to (6);
						\draw (6) to (7);
						\draw (7) to (10);
						\draw (10) to (9);
					\end{pgfonlayer}
				\end{tikzpicture}
				
				\caption{The graph $G(2,4)$.}
				\label{fig:g_2_example}
			\end{figure}
			
			\textit{Case 2:} $\delta = 3$.
			For $k\geq 3$, we construct a graph $G(3,k)$.
			Let the graph have vertices $x$ and $y$, as well as vertices for all ordered pairs $(i,j)$, where $i \in [1,k]$ and $j \in [1,3]$.
			Let $x$ and $y$ be adjacent to each other, and to every vertex of the form $(1,j)$.
			For $i\geq 2$, let the vertices $(i,1)$, $(i,2)$ and $(i,3)$ all be adjacent. 
			For $1 \leq i \leq k-1$, let $G(3,k)$ have all edges of the form $(i,j)(i+1, j)$.
			For $1 \leq i \leq k-1$ and $j\in \{2,3\}$ add all edges $(i,j)(i+1,j-1)$ (see Figure \ref{fig:g_3_example} for an example where $k=4$).
			The resulting graph $G(3,k)$ has order $n = 3k+2$.
			The vertices $(1,1)$ and $(k,3)$ have degree 3, every other vertex has higher degree. 
			It is routine to check that the singletons $\{x\}$ and $\{y\}$ are PDSs with propagation radius $3(k-1) + 1 = n - \delta - 1$.
			Per Lemma \ref{lem_fort}, it suffices to show that no neighbour of $x$ or $y$ is a PDS. 
			By inspection, the singleton $\{(1,j)\}$ is not a PDS for all $j\in \{1,2,3\}$.
			
			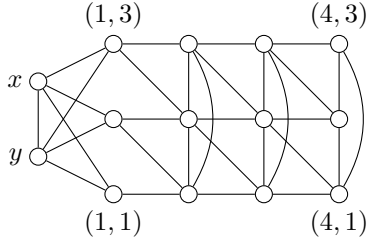
\begin{figure}[h!]
				\centering
				\begin{tikzpicture}[scale=0.5, inner sep=0.8mm]
					\begin{pgfonlayer}{nodelayer}
						\node [style=whitevertex, label={left:$x$}] (0) at (-4, 1) {};
						\node [style=whitevertex, label={left:$y$}] (1) at (-4, -1) {};
						\node [style=whitevertex] (2) at (-2, 0) {};
						\node [style=whitevertex, label={above:$(1,3)$}] (3) at (-2, 2) {};
						\node [style=whitevertex, label={below:$(1,1)$}] (4) at (-2, -2) {};
						\node [style=whitevertex] (5) at (0, 2) {};
						\node [style=whitevertex] (6) at (0, 0) {};
						\node [style=whitevertex] (7) at (0, -2) {};
						\node [style=whitevertex] (8) at (2, 2) {};
						\node [style=whitevertex] (9) at (2, -2) {};
						\node [style=whitevertex] (10) at (2, 0) {};
						\node [style=whitevertex, label={above:$(4,3)$}] (11) at (4, 2) {};
						\node [style=whitevertex] (12) at (4, 0) {};
						\node [style=whitevertex, label={below:$(4,1)$}] (13) at (4, -2) {};
					\end{pgfonlayer}
					\begin{pgfonlayer}{edgelayer}
						\draw (0) to (1);
						\draw (1) to (2);
						\draw (2) to (0);
						\draw (0) to (3);
						\draw (1) to (4);
						\draw (0) to (4);
						\draw (1) to (3);
						\draw (3) to (5);
						\draw (5) to (8);
						\draw (8) to (11);
						\draw (6) to (2);
						\draw (6) to (10);
						\draw (10) to (12);
						\draw (4) to (7);
						\draw (7) to (9);
						\draw (9) to (13);
						\draw [bend left] (5) to (7);
						\draw [bend left] (8) to (9);
						\draw [bend left] (11) to (13);
						\draw (5) to (6);
						\draw (6) to (7);
						\draw (8) to (10);
						\draw (10) to (9);
						\draw (11) to (12);
						\draw (12) to (13);
						\draw (2) to (7);
						\draw (3) to (6);
						\draw (5) to (10);
						\draw (6) to (9);
						\draw (8) to (12);
						\draw (10) to (13);
					\end{pgfonlayer}
				\end{tikzpicture}
				
				\caption{The graph $G(3,4)$.}
				\label{fig:g_3_example}
			\end{figure}
			
			\textit{Case 4:} $\delta \geq 4$.
			Let $k\geq 3$, and construct $G(\delta, k)$ as follows.
			Let $G(\delta, k)$ have vertices $u,v,x$ and $y$, as well as a vertex for each ordered pair $(i,j)$, where $i\in [1,k]$ and $j\in [1,\delta-2]$.
			Let all four vertices $u,v,x$ and $y$ be adjacent; have $u,v$ adjacent to all vertices $(1,j)$; have $x,y$ adjacent to all vertices $(k,j)$; let all vertices $(i,1), (i,2), \dots, (i,\delta-2)$ be adjacent to each other; for $1\leq i \leq k-1$ include all possible edges of the form $(i,j)(i+1,j)$; for $1\leq i \leq k-1$ and $2 \leq j \leq \delta - 2$ add all edges $(i,j)(i+1, j-1)$ (see Figure \ref{fig:g_4_example} for the case where $\delta=4$ and $k=5$).
			The graph $G(\delta, k)$ has $n = (\delta-2) k + 4$ vertices, and minimum degree $\delta$ (attained by the vertex $(1,1)$).
			Each of the singletons $\{u\}, \{v\}, \{x\}$ and $\{y\}$ is a PDS with propagation radius $(k-1)(\delta - 2) + 1 = n-\delta -1$.
			Note that $N(u)  \setminus  \{v\} = N(v)  \setminus  \{u\}$, and that $N(x)  \setminus  \{y\} = N(y)  \setminus  \{x\}$, so by Lemma \ref{lem_fort}, no other singleton is a PDS.
			
			\begin{figure}[ht]
				\centering
				\begin{tikzpicture}[scale=0.75]
					\begin{pgfonlayer}{nodelayer}
						\node [style=whitevertex] (0) at (0, 1) {};
						\node [style=whitevertex] (1) at (0, -1) {};
						\node [style=whitevertex] (2) at (-2, 1) {};
						\node [style=whitevertex] (3) at (-2, -1) {};
						\node [style=whitevertex] (4) at (2, 1) {};
						\node [style=whitevertex] (5) at (2, -1) {};
						\node [style=whitevertex] (6) at (4, 1) {};
						\node [style=whitevertex] (7) at (4, -1) {};
						\node [style=whitevertex] (8) at (-4, 1) {};
						\node [style=whitevertex] (9) at (-4, -1) {};
						\node [style=whitevertex] (10) at (6, 0.75) {};
						\node [style=whitevertex] (11) at (6, -0.75) {};
						\node [style=whitevertex] (12) at (-6, 0.75) {};
						\node [style=whitevertex] (13) at (-6, -0.75) {};
						\node [style=none] (14) at (-6.75, 0.75) {$u$};
						\node [style=none] (15) at (-6.75, -0.75) {$v$};
						\node [style=none] (16) at (-4, 1.75) {$(1,2)$};
						\node [style=none] (17) at (-4, -1.75) {$(1,1)$};
						\node [style=none] (18) at (4, 1.75) {$(5,2)$};
						\node [style=none] (19) at (4, -1.75) {$(5,1)$};
						\node [style=none] (20) at (-4, 2.75) {};
						\node [style=none] (21) at (4, 2.75) {};
						\node [style=none] (22) at (-4, -2.75) {};
						\node [style=none] (23) at (4, -2.75) {};
						\node [style=none] (24) at (-4, 3.25) {};
						\node [style=none] (25) at (5.75, 2.75) {};
						\node [style=none] (26) at (5.75, -2.75) {};
						\node [style=none] (27) at (-4, -3.25) {};
						\node [style=none] (28) at (7.25, 0.75) {$x$};
						\node [style=none] (29) at (7.25, -0.75) {$y$};
					\end{pgfonlayer}
					\begin{pgfonlayer}{edgelayer}
						\draw (12) to (13);
						\draw (12) to (8);
						\draw (12) to (9);
						\draw (13) to (9);
						\draw (8) to (9);
						\draw (2) to (3);
						\draw (0) to (1);
						\draw (4) to (5);
						\draw (6) to (7);
						\draw (10) to (11);
						\draw (8) to (2);
						\draw (2) to (0);
						\draw (0) to (4);
						\draw (4) to (6);
						\draw (6) to (10);
						\draw (11) to (7);
						\draw (7) to (5);
						\draw (5) to (1);
						\draw (1) to (3);
						\draw (3) to (9);
						\draw (13) to (8);
						\draw (8) to (3);
						\draw (2) to (1);
						\draw (0) to (5);
						\draw (4) to (7);
						\draw (7) to (10);
						\draw (6) to (11);
						\draw [bend left] (12) to (20.center);
						\draw [bend left=15, looseness=0.75] (20.center) to (21.center);
						\draw [bend left] (21.center) to (10);
						\draw [bend right=15, looseness=0.75] (22.center) to (23.center);
						\draw [bend right] (13) to (22.center);
						\draw [bend right] (23.center) to (11);
						\draw [bend left] (12) to (24.center);
						\draw [bend left, looseness=0.50] (24.center) to (25.center);
						\draw [bend left=45] (25.center) to (11);
						\draw [bend right] (13) to (27.center);
						\draw [bend right, looseness=0.50] (27.center) to (26.center);
						\draw [bend right=45] (26.center) to (10);
					\end{pgfonlayer}
				\end{tikzpicture}
				
				\caption{The graph $G(4,5)$. Vertices $u,v$ are on the left, and $x,y$ on the right.}
				\label{fig:g_4_example}
			\end{figure}
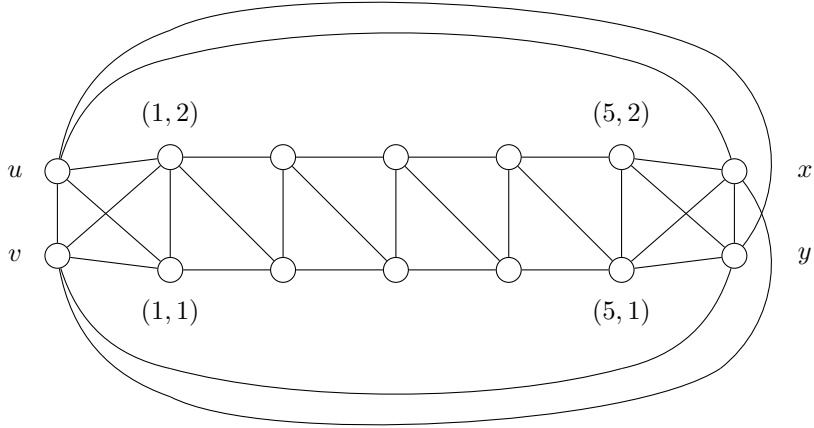
		\end{proof}
		
	\end{thm}
	
	\section{Sharpness for regular graphs}
	
	\label{sec:sharp_reg}
	
	In this section we show that the bound of Theorem \ref{thm:delta_1_up_bound} can not be improved for regular graphs by constructing for each $\delta \geq 2$ a regular graph $G(\delta)$ that attains the bound $rad_p(G(\delta)) = n - \delta - 1$ of Theorem \ref{thm:delta_1_up_bound}.
	
	\begin{prop}
		\label{prop:odd_regular_sharp}
		For each integer $\delta \geq 2$ there exists a $\delta$-regular graph, $G(\delta)$ such that $rad_p(G(\delta))=n-\delta-1$.
	\end{prop}
	\begin{proof}
		Let $n=\delta+3$. Let $K_n$ and $C_n$ have vertex set $V=\{v_1,...,v_n\}$. Let $C_n$ have edge set $E(C_n)=\{v_1v_2, v_2v_3,...,v_nv_1\}$. We define $G(\delta)$ as follows: $V(G(\delta))=V(K_n)$ and $E(G(\delta))=E(K_n)-E(C_n)$. Note that $G(\delta)$ is the circulant $C_n\langle 2,3,\dots,\lceil\frac{n}{2}\rceil \rangle$. Since $\delta \geq 2$, we must have $n \geq 5$. Let $S=\{v_3\}$. Then $N[S]=V-\{v_2,v_4\}$. Since $v_1 \in N[S]$ is adjacent to $v_4$ and $v_1v_2 \not\in E(G(\delta))$, $v_1$ propagates to $v_4$ in step 2. By symmetry we also have that $v_5$ propagates to $v_2$ in step 2. Thus, $S$ is a $\gamma_p(G(\delta))$-set with propagation radius 2. By the symmetry of $G(\delta)$, every singleton is a PDS with propagation radius 2. Therefore $rad_p(G(\delta))=2=n-\delta-1$.
	\end{proof}
	
	\begin{figure}[hbt!]
		\centering
		\begin{tikzpicture}[scale=0.8, inner sep = 0.8mm]
			\begin{pgfonlayer}{nodelayer}

				\node [style=whitevertex] (1) at (2.75, 3) {};
				\node [style=whitevertex] (2) at (5.25, 3) {};
				\node [style=whitevertex] (3) at (7, 1) {};
				\node [style=whitevertex] (4) at (6.5, -1.25) {};
				\node [style=whitevertex] (5) at (4, -2.25) {};
				\node [style=whitevertex] (6) at (1.5, -1.25) {};
				\node [style=whitevertex] (7) at (1, 1) {};
			\end{pgfonlayer}
			\begin{pgfonlayer}{edgelayer}

				\draw (1) to (3);
				\draw (1) to (4);
				\draw (1) to (5);
				\draw (1) to (6);
				\draw (2) to (4);
				\draw (2) to (5);
				\draw (2) to (6);
				\draw (2) to (7);
				\draw (3) to (5);
				\draw (3) to (6);
				\draw (3) to (7);
				\draw (4) to (6);
				\draw (4) to (7);
				\draw (5) to (7);
			\end{pgfonlayer}
		\end{tikzpicture}
		
		\caption{A regular graph $G(4)$ attaining the $n-\delta - 1$ bound. The graph has $\delta = 4$ and $n=7$.}
		\label{fig:extremal_regular_graphs}
	\end{figure}
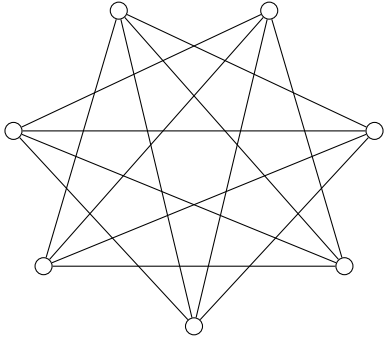
	
	\section{Upper bounds when $\bm{\gamma_p > 1}$}
	
	\label{sec:upp_gam_2}
	
	In this section, we show that the upper bound in Theorem \ref{thm:delta_1_up_bound} can be improved when $\gamma_p > 1$ --- and that the improved bound is sharp for many values of $\delta$ and $\gamma_p$.
	
	Let $S$ be a power dominating set with propagation radius $k$ and $P^0 = \emptyset, P^1 = N[S], \dots, P^k$ be its propagation sequence.
	Define the \textbf{excess} $\bm{\epsilon(S)}$ of $S$ as the sum
	\[
	\epsilon(S) = \sum_{i=1}^k \left(|P^i  \setminus  P^{i-1}| - 1 \right).
	\]
	At each propagation step of $P$, at least one new vertex is monitored. 
	The excess measures how many extra vertices are monitored, totalled over all the steps, which yields Remark \ref{rem:excess}.
	
	\begin{rem}
		Let $G$ be a graph of order $n$.
		Let $S$ be a $\gamma_p$-set of minimum radius (and thus maximum excess) in $G$.
		Then $rad_p(G) = n - \epsilon(S)$.
		\label{rem:excess}
	\end{rem}
	
	\begin{thm}
		Let $G$ be a connected graph of order $n$, minimum degree $\delta \geq 2$ and power domination number $\gamma_p \geq 2$. 
		Then
		\[
		rad_p(G) \leq 
		\begin{cases}
			n - 5 & \text{ if } \gamma_p = 2, \delta = 2\\
			n - \delta - 2 & \text{ if } \gamma_p = 2, \delta > 2\\
			n - \gamma_p - \max\{2\gamma_p, \gamma_p + \delta\} + 1 & \text{ if } \gamma_p > 2\\
		\end{cases}
		\]
		\label{thm:upper_bound_geq_2}
	\end{thm}
	
	\begin{proof}
		Let $G=(V,E)$ be as hypothesised in the Theorem statement, and let $S$ be a power dominating set with with $\gamma_p$ vertices and propagation sequence $P^0, P^1, \dots, P^k$, where $k=rad_p(G)$ --- and note that $S$ has maximum excess among all $\gamma_p$-power dominating sets.
		Assume to the contrary that either $rad_p(G) > n - \gamma_p - \max\{2\gamma_p, \gamma_p + \delta\} + 1$, or that $rad_p(G) > n - \delta - 2$ in the case that $\gamma_p =2$ and $\delta > 2$.
		Per Remark \ref{rem:excess}, this is equivalent to assuming that $\epsilon(S) < \gamma_p + \max\{2\gamma_p, \gamma_p + \delta\} - 1$, or that $\epsilon(S) < \delta + 2$ when $\gamma_p = 2$ and $\delta > 2$.
		
		Partition $S$ into two parts, $S_c$ and $S_i$: $S_c$ contains the vertices of $S$ with at least one neighbour in $S$, and $S_i$ contains the vertices of $S$ that have no neighbours in $S$.
		Note that $|S_c| + |S_i| = \gamma_p$, and that $|S_c| \neq 1$.
		
		\textit{Claim 1:} Every vertex of $S_i$ has at least one external $S$-private neighbour.
		Assume to the contrary that there is a vertex $u\in S_i$ with no external $S$-private neighbour.
		If there is no vertex $z \in V \setminus N[S]$ such that $N(z)\cap N(u) \neq \emptyset$, then $S \setminus \{u\}$ is a smaller power dominating set, contradicting the minimality of $S$.
		If there is such a vertex $z$, let $y$ be a vertex of $N(z)\cap N(u)$. 
		The set $S' = (S \setminus \{u\}) \cup \{y\}$ is a power-dominating set with $\gamma_p$ vertices, and $\epsilon(S) < \epsilon(S')$, contradicting the choice of $S$ and proving the claim.
		
		\textit{Claim 2:} Every vertex of $S_c$ has at least two $S$-private neighbours that do not belong to $S$.
		Assume to the contrary some vertex $u\in S_c$ has at most one $S$-private neighbour outside of $S$, and note that $u$ is adjacent to some vertex $v\in S$.
		Then $S \setminus \{u\}$ is a smaller power dominating set, contradicting the minimality of $S$.
		
		We partition $N(S)$ into three parts $T_c, T_i$ and $T_r$.
		For each vertex in $S_c$, pick two of its external $S$-private neighbours, and let $T_c$ consist of all the vertices picked in this way.
		For each vertex in $S_i$, pick one of its external $S$-private neighbours, and let $T_i$ consist of these vertices (see Figure \ref{fig:tc_ti_tr}).
		Let $T_r = N(S)  \setminus  (T_c  \cup  T_i)$ contain all the remaining vertices of $N(S)$.
		Note that $|T_c| = 2|S_c|$ and $|T_i| = |S_i|$.
		
		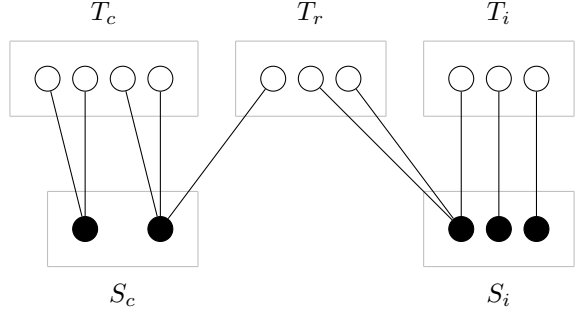
\begin{figure}[ht]
			\centering
			\begin{tikzpicture}[scale=0.5]
				\begin{pgfonlayer}{nodelayer}
					\node [style=none] (0) at (-7, 1) {};
					\node [style=none] (1) at (-7, -1) {};
					\node [style=none] (2) at (-3, -1) {};
					\node [style=none] (3) at (-3, 1) {};
					\node [style=none] (4) at (3, 1) {};
					\node [style=none] (5) at (3, -1) {};
					\node [style=none] (6) at (7, 1) {};
					\node [style=none] (7) at (7, -1) {};
					\node [style=none] (8) at (-8, 5) {};
					\node [style=none] (9) at (-8, 3) {};
					\node [style=none] (10) at (-3, 3) {};
					\node [style=none] (11) at (-3, 5) {};
					\node [style=none] (12) at (3, 5) {};
					\node [style=none] (13) at (3, 3) {};
					\node [style=none] (14) at (7, 5) {};
					\node [style=none] (15) at (7, 3) {};
					\node [style=none] (16) at (-2, 5) {};
					\node [style=none] (17) at (-2, 3) {};
					\node [style=none] (18) at (2, 5) {};
					\node [style=none] (19) at (2, 3) {};
					\node [style=blackvertex] (20) at (-4, 0) {};
					\node [style=blackvertex] (21) at (-6, 0) {};
					\node [style=whitevertex] (22) at (-4, 4) {};
					\node [style=whitevertex] (23) at (-5, 4) {};
					\node [style=whitevertex] (24) at (-6, 4) {};
					\node [style=whitevertex] (25) at (-7, 4) {};
					\node [style=blackvertex] (26) at (4, 0) {};
					\node [style=blackvertex] (27) at (5, 0) {};
					\node [style=blackvertex] (28) at (6, 0) {};
					\node [style=whitevertex] (29) at (4, 4) {};
					\node [style=whitevertex] (30) at (5, 4) {};
					\node [style=whitevertex] (31) at (6, 4) {};
					\node [style=whitevertex] (32) at (-1, 4) {};
					\node [style=whitevertex] (33) at (0, 4) {};
					\node [style=whitevertex] (34) at (1, 4) {};
					\node [style=none] (35) at (-5, -1.75) {$S_c$};
					\node [style=none] (36) at (5, -1.75) {$S_i$};
					\node [style=none] (37) at (-5.5, 5.75) {$T_c$};
					\node [style=none] (38) at (0, 5.75) {$T_r$};
					\node [style=none] (39) at (5, 5.75) {$T_i$};
				\end{pgfonlayer}
				\begin{pgfonlayer}{edgelayer}
					\draw [style=light] (8.center) to (9.center);
					\draw [style=light] (9.center) to (10.center);
					\draw [style=light] (10.center) to (11.center);
					\draw [style=light] (11.center) to (8.center);
					\draw [style=light] (0.center) to (1.center);
					\draw [style=light] (1.center) to (2.center);
					\draw [style=light] (2.center) to (3.center);
					\draw [style=light] (3.center) to (0.center);
					\draw [style=light] (4.center) to (5.center);
					\draw [style=light] (5.center) to (7.center);
					\draw [style=light] (7.center) to (6.center);
					\draw [style=light] (6.center) to (4.center);
					\draw [style=light] (12.center) to (13.center);
					\draw [style=light] (13.center) to (15.center);
					\draw [style=light] (15.center) to (14.center);
					\draw [style=light] (14.center) to (12.center);
					\draw [style=light] (16.center) to (17.center);
					\draw [style=light] (17.center) to (19.center);
					\draw [style=light] (19.center) to (18.center);
					\draw [style=light] (18.center) to (16.center);
					\draw (25) to (21);
					\draw (21) to (24);
					\draw (23) to (20);
					\draw (20) to (22);
					\draw (29) to (26);
					\draw (30) to (27);
					\draw (31) to (28);
					\draw (20) to (32);
					\draw (26) to (33);
					\draw (26) to (34);
				\end{pgfonlayer}
			\end{tikzpicture}
			
			\caption{Partition of $N[S]$ into five parts. The vertices in parts of $S$ are black, and the vertices in $N(S)$ are white.}
			\label{fig:tc_ti_tr}
		\end{figure}
		
		\textit{Case 1:} $\gamma_p > 2$ and $\delta \geq \gamma_p$. 
		Because $rad_p(G) > n - 2\gamma_p - \delta + 1$, we have $|S| + |N(S)| - 1 \leq \epsilon(S) < 2\gamma_p + \delta - 1$.
		Thus $|N(S)| < \gamma_p + \delta$.
		Partitioning $N(S)$ and using $|S_c| + |S_i| = \gamma_p$ yields:
		\begin{align*}
			|T_c| + |T_i| + |T_r| &< \gamma_p + \delta\\
			2|S_c| + |S_i| + |T_r| &< \gamma_p + \delta\\
			|S_c| + |T_r| &< \delta
		\end{align*}
		We claim either $S_c = \emptyset$ or $S_i = \emptyset$.
		Assume not, and let $u$ be a vertex of $S_i$. 
		Then $|T_r| < \delta - |S_c| \leq \delta - 2$. 
		The vertex $u$ has one neighbour in $T_i$, and all other neighbours in $T_r$.
		Thus we have the contradiction $d(u) < \delta$, proving the claim.
		Suppose $S_c = \emptyset$, so $S = S_i$, and let $u$ be a vertex of $S_i$.
		Since $d(u) \geq \delta$ and $|T_r| < \delta$, we have $|T_r| = \delta - 1$, and $u$ is adjacent to every vertex of $T_r$.
		Similarly, every vertex of $S$ is adjacent to every vertex of $T_r$.
		Let $v$ be a vertex of $T_r$, and set $S' = \{u,v\}$.
		The set $N(S')$ contains all of $S \cup T_r$. 
		Since each vertex of $S$ can propagate to its unique neighbour in $T_i$, we have that $N[S] \subseteq P^2(S')$.
		Thus $S'$ is a power dominating set with $2 < \gamma_p$ vertices, contradicting the choice of $S$.
		Thus $S = S_c$.
		Since $|S_c| + |T_r| \leq \delta - 1$, we have $\gamma_p \leq \delta - 1$, and $|T_r| \leq \delta - \gamma_p - 1$.
		Consider an arbitrary vertex $u\in S$. 
		The only possible neighbours of $u$ are the two neighbours in $T_c$, the $|T_r| \leq \delta - \gamma_p - 1$ vertices in $T_r$, and the remaining $\gamma_p - 1$ vertices of $S$. 
		There are at most $\delta$ of these possible neighbours in total.
		Thus $u$ must be adjacent to every other vertex of $S$, and is adjacent to every vertex of $T_r$ --- of which there are exactly $\delta - \gamma_p - 1$.
		As $u$ is arbitrary, the set $S$ is a clique, and every vertex of $S$ is adjacent to every vertex of $T_r$.
		Denote by $x$ and $y$ the neighbours of $u$ in $T_c$.
		We claim that $x$ is not adjacent to any vertex of $T_c \setminus \{y\}$.
		Assume to the contrary that $x$ is adjacent to some vertex $z$ of $T_c  \setminus  \{y\}$, and let $w$ denote the neighbour of $z$ in $S$.
		Let $S' = (S  \setminus  \{u,w\}) \cup \{x\}$.
		The set $N[S']$ contains all of $N[S]$ except for possibly $y$, and one vertex $q$ of $N(w) \cap T_c$. 
		Therefore $N[S] \subseteq P^2(S')$, so $S'$ is a power dominating set.
		This contradicts the choice of $S$ and proves the claim.
		The degree of $x$ is at least $\delta$, $x$ is not adjacent to any vertex of $T_c$ (possibly except for $y$), and the only vertex of $S_c$ to which $x$ is adjacent is $u$.
		Thus $x$ has at most $\delta - \gamma_p + 1$ neighbours in $N[S]$. 
		As $\gamma_p \geq 3$, we see $x$ has at least two neighbours in $V  \setminus  N[S]$.
		Therefore $S' = (S  \setminus  \{u\}) \cup \{x\}$ is a power dominating set with $\epsilon(S') > \epsilon(S)$, a contradiction.
		
		\textit{Case 2:} $\gamma_p > 2$ and $\delta < \gamma_p$.
		By assumption $rad_p(G) > n - 3\gamma_p + 1$, thus $|S| + |N(S)| - 1 \leq \epsilon(S) < 3\gamma_p - 1$.
		Therefore $|N(S)| \leq 2\gamma_p - 1$.
		We claim $|S_c|\leq \gamma_p - 2$.
		If $S_c = S$, then $|T_c| = 2|S_c| = 2\gamma_p$, contradicting $|N(S)| \leq 2\gamma_p - 1$. 
		If $|S_c| = \gamma_p - 1$, then $|S_i| = 1$, and the one vertex of $S_i$ has one neighbour in $T_i$ and at least one neighbour in $T_r$, so $|T_c| + |T_i| + |T_r| \geq 2|S_c| + 2 \geq 2\gamma_p$, a contradiction proving the claim.
		Thus $|S_i| \geq 2$.
		Each vertex of $|S_i|$ has at least one neighbour in $T_r$, so $T_r \neq \emptyset$.
		Recall that $|S_c| = \gamma_p - |S_i|$ to find:
		\begin{align*}
			|T_c| + |T_i| + |T_r| &\leq 2\gamma_p - 1\\
			2|S_c| + |S_i| + |T_r| &\leq 2\gamma_p - 1\\
			2(\gamma_p - |S_i|) + |S_i| + |T_r| &\leq 2\gamma_p - 1\\
			|T_r| + 1 &\leq |S_i|.
		\end{align*}
		Let $S' = S_c \cup T_r$, and note that $S'$ is a power dominating set.
		Since $|T_r| < |S_i|$, we have $|S'| < |S|$, contradicting the choice of $S$.
		
		\textit{Case 3:} $\gamma_p = \delta = 2$.
		By the assumption that $rad_p(G) > n - 5$, we have that $|N[S]| - 1 \leq \epsilon(S) < 5$, so $|N(S)| < 4$.
		If $S_c \neq \emptyset$, then $S = S_c$, and so $|T_c| = 2|S_c| = 4$ which contradicts $|N(S)| < 4$. 
		Thus $S_c = \emptyset$ and $S = S_i$.
		Note that each vertex of $S$ has at least one neighbour in $T_r$.
		If $|T_r| \geq 2$, then $|N(S)| \geq 4$.
		If $|T_r| = 1$, then the single vertex of $T_r$ is a power dominating set, contradicting the choice of $S$.
		
		\textit{Case 4:} $\gamma_p = 2$ and $\delta \geq 3$.
		Because $rad_p(G) > n - \delta - 2$, we have $|N[S]| - 1 \leq \epsilon(S) < \delta + 2$.
		Since $|S| = 2$, we get $|N(S)| \leq \delta$.
		Since $\gamma_p = 2$, either $S = S_c$ or $S = S_i$. 
		If $S = S_i$, then $|T_i| = 2$, and each vertex of $S_i$ has at least $\delta - 1$ neighbours in $T_r$. 
		Thus $|N(S)| \geq 2 + (\delta - 1) > \delta$, a contradiction.
		Thus $S = S_c$, and $|T_c| = 2|S_c| = 4$.
		Each vertex of $S_c$ has at least $\delta - 3$ neighbours in $T_r$, so $|N(S)| \geq 4 + (\delta - 3) > \delta$, a contradiction.
	\end{proof} 
	
	The bounds in Theorem \ref{thm:upper_bound_geq_2} are sharp when $\delta = 2$ or $\gamma_p = 2$. 
	To show the bound is sharp for $\delta = 2$ and $\gamma_p \geq 2$, consider the graph $D(k)$ formed as follows: Take a path $v_1, v_2, \dots, v_k$ and attach to each vertex $v_i$ a pair of adjacent vertices. 
	
	\begin{figure}[ht]
		\centering
		\begin{tikzpicture}[scale=0.7]
			\begin{pgfonlayer}{nodelayer}
				\node [style=blackvertex] (0) at (0, 0) {};
				\node [style=blackvertex] (1) at (-2, 0) {};
				\node [style=blackvertex] (2) at (-4, 0) {};
				\node [style=blackvertex] (3) at (2, 0) {};
				\node [style=blackvertex] (4) at (4, 0) {};
				\node [style=whitevertex] (5) at (-4.5, 1) {};
				\node [style=whitevertex] (6) at (-3.5, 1) {};
				\node [style=whitevertex] (7) at (-2.5, 1) {};
				\node [style=whitevertex] (8) at (-1.5, 1) {};
				\node [style=whitevertex] (9) at (-0.5, 1) {};
				\node [style=whitevertex] (10) at (0.5, 1) {};
				\node [style=whitevertex] (11) at (1.5, 1) {};
				\node [style=whitevertex] (12) at (2.5, 1) {};
				\node [style=whitevertex] (13) at (3.5, 1) {};
				\node [style=whitevertex] (14) at (4.5, 1) {};
			\end{pgfonlayer}
			\begin{pgfonlayer}{edgelayer}
				\draw (2) to (1);
				\draw (1) to (0);
				\draw (0) to (3);
				\draw (3) to (4);
				\draw (5) to (2);
				\draw (2) to (6);
				\draw (6) to (5);
				\draw (7) to (8);
				\draw (8) to (1);
				\draw (1) to (7);
				\draw (9) to (10);
				\draw (10) to (0);
				\draw (0) to (9);
				\draw (11) to (12);
				\draw (12) to (3);
				\draw (3) to (11);
				\draw (13) to (14);
				\draw (14) to (4);
				\draw (4) to (13);
			\end{pgfonlayer}
		\end{tikzpicture}
		
		\caption{The graph $D(5)$. The black vertices form a minimum power dominating set.}
		\label{fig:d_5}
	\end{figure}
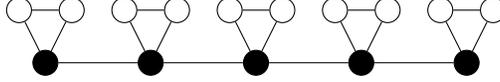
	
	The graph $D(k)$ has $3k$ vertices, minimum degree $2$, and the $k$ vertices $\{v_1, \dots, v_k\}$ are a minimum power dominating set. 
	Thus the power domination radius of $D(k)$ is $1 = 3k - k - 2k + 1$.
	
	To show the bound is sharp for $\gamma_p = 2$, we use the join to recursively construct a family of graphs $F(\delta)$. 
	
	\begin{lem}
		Let $G$ and $H$ be graphs with no isolated vertices such that $\gamma_p(G) \geq 2$ and $\gamma_p(H) \geq 2$.
		Denote $\delta(G) = \delta_G$, $\delta(H) = \delta_H$, and let $n_G$ and $n_H$ be the orders of $G$ and $H$, respectively. Then
		\begin{itemize}
			\item $\gamma_p(G+H) = 2$,
			\item $rad_p(G+H) = 1$,
			\item $G+H$ has order $n_G + n_H$,
			\item $\delta(G+H) = \min\{\delta_G + n_H, \delta_H + n_G\}$.
		\end{itemize}
		\label{lem:join_gamma_2}
	\end{lem}
	
	\begin{proof}
		We claim that if $S \subseteq V(G)$ is a power dominating set of $G+H$, then it is also a power dominating set of $G$. 
		Suppose $S$ is a power dominating set of $G+H$. 
		If $S$ dominates $G+H$, it also dominates $G$, so we may assume there are vertices of $G$ not adjacent to $S$.
		Let $u_1 \to v_1, u_2 \to v_2, \dots, u_k \to v_k$ be a forcing sequence for $N[S]$.
		As $V(H) \subset N[S]$, we have that $v_i \in V(G)$ for all $1 \leq i \leq k$. 
		Since $G$ has no isolated vertices, $v_k$ has some neighbour $w \in V(G)$. 
		Each vertex of $H$ is adjacent to every (unmonitored) vertex of $G$, so $u_i \in V(G)$ for every $i < k$. 
		Thus the sequence $u_1 \to v_1, u_2 \to v_2, \dots, u_{k-1} \to v_{k-1}, w \to v_k$ is a forcing sequence for $N[S]$ in $G$, which proves the claim.
		
		Note that no singleton $\{x\}$ containing a vertex of $G$ is a power dominating set of $G+H$: if $\{x\}$ is a power dominating set of $G+H$, then by the claim it is also a power dominating set in $G$, contradicting $\gamma_p(G) \geq 2$. 
		Similarly, no singleton of $H$ is a power dominating set of $G+H$, so $\gamma_p(G+H) \geq 2$.
		Let $u$ be a vertex of $G$ and $v$ a vertex of $H$. 
		It's clear that $\{u,v\}$ is a dominating set of $G+H$, so $\gamma(G+H) = \gamma_p(G+H) = 2$, and $rad_p(G+H) = 1$.
		The reader will easily check the order and minimum degree of $G+H$.
	\end{proof}
	
	We construct a family $F(\delta)$, $\delta \neq 2$ of graphs such that $F(\delta)$ has minimum degree $\delta$ and order $\delta + 3$. 
	Let $F(0) = \overline{K_3}$, $F(1) = 2K_2$, $F(3) = K_{3,3}$ and $F(6) = \overline{K_3} + K_{3,3}$.
	For all positive integers $\delta\neq 2$, define $F(\delta + 4) = (2K_2) + F(\delta)$.
	Note that there is no suitable candidate graph for $F(2)$, hence the sequence of graphs $F(\delta)$ with $\delta \equiv 2 \pmod 4$ begins with $F(6)$.
	
	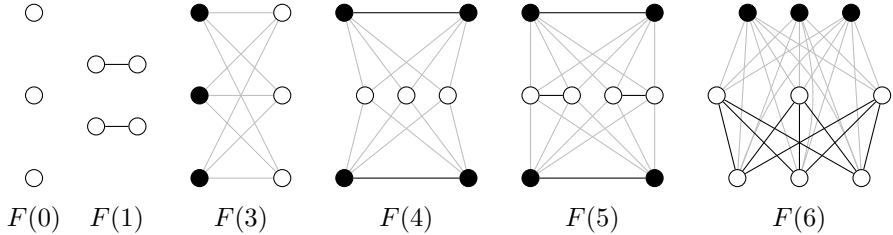
\begin{figure}[ht]
		\centering
		\begin{tikzpicture}[scale=0.55, inner sep=0.8mm]
			\begin{pgfonlayer}{nodelayer}
				\node [style=whitevertex] (0) at (-5, 2) {};
				\node [style=whitevertex] (1) at (-5, 0) {};
				\node [style=whitevertex] (2) at (-5, -2) {};
				\node [style=whitevertex] (3) at (-3.5, 0.75) {};
				\node [style=whitevertex] (4) at (-2.5, 0.75) {};
				\node [style=whitevertex] (5) at (-3.5, -0.75) {};
				\node [style=whitevertex] (6) at (-2.5, -0.75) {};
				\node [style=blackvertex] (7) at (-1, 2) {};
				\node [style=blackvertex] (8) at (-1, 0) {};
				\node [style=blackvertex] (9) at (-1, -2) {};
				\node [style=whitevertex] (10) at (1, 2) {};
				\node [style=whitevertex] (11) at (1, 0) {};
				\node [style=whitevertex] (12) at (1, -2) {};
				\node [style=blackvertex] (13) at (2.5, 2) {};
				\node [style=blackvertex] (14) at (5.5, 2) {};
				\node [style=blackvertex] (15) at (2.5, -2) {};
				\node [style=blackvertex] (16) at (5.5, -2) {};
				\node [style=whitevertex] (17) at (3, 0) {};
				\node [style=whitevertex] (18) at (4, 0) {};
				\node [style=whitevertex] (19) at (5, 0) {};
				\node [style=blackvertex] (20) at (7, 2) {};
				\node [style=blackvertex] (21) at (10, 2) {};
				\node [style=blackvertex] (22) at (7, -2) {};
				\node [style=blackvertex] (23) at (10, -2) {};
				\node [style=whitevertex] (24) at (7, 0) {};
				\node [style=whitevertex] (25) at (8, 0) {};
				\node [style=whitevertex] (26) at (10, 0) {};
				\node [style=whitevertex] (27) at (9, 0) {};
				\node [style=whitevertex] (28) at (12, -2) {};
				\node [style=whitevertex] (29) at (13.5, -2) {};
				\node [style=whitevertex] (30) at (15, -2) {};
				\node [style=whitevertex] (31) at (11.5, 0) {};
				\node [style=whitevertex] (32) at (13.5, 0) {};
				\node [style=whitevertex] (33) at (15.5, 0) {};
				\node [style=blackvertex] (34) at (13.5, 2) {};
				\node [style=blackvertex] (35) at (12.25, 2) {};
				\node [style=blackvertex] (36) at (14.75, 2) {};
				\node [style=none] (37) at (-5, -3) {$F(0)$};
				\node [style=none] (38) at (-3, -3) {$F(1)$};
				\node [style=none] (39) at (0, -3) {$F(3)$};
				\node [style=none] (40) at (4, -3) {$F(4)$};
				\node [style=none] (41) at (8.5, -3) {$F(5)$};
				\node [style=none] (42) at (13.5, -3) {$F(6)$};
			\end{pgfonlayer}
			\begin{pgfonlayer}{edgelayer}
				\draw (5) to (6);
				\draw (3) to (4);
				\draw [style=light] (7) to (10);
				\draw [style=light] (8) to (11);
				\draw [style=light] (9) to (12);
				\draw [style=light] (7) to (12);
				\draw [style=light] (10) to (9);
				\draw [style=light] (7) to (11);
				\draw [style=light] (8) to (10);
				\draw [style=light] (11) to (9);
				\draw [style=light] (8) to (12);
				\draw (13) to (14);
				\draw (15) to (16);
				\draw [style=light] (17) to (15);
				\draw [style=light] (17) to (16);
				\draw [style=light] (16) to (18);
				\draw [style=light] (18) to (15);
				\draw [style=light] (15) to (19);
				\draw [style=light] (19) to (16);
				\draw [style=light] (13) to (17);
				\draw [style=light] (17) to (14);
				\draw [style=light] (13) to (18);
				\draw [style=light] (18) to (14);
				\draw [style=light] (13) to (19);
				\draw [style=light] (19) to (14);
				\draw (20) to (21);
				\draw (22) to (23);
				\draw [style=light] (24) to (22);
				\draw [style=light] (24) to (23);
				\draw [style=light] (23) to (25);
				\draw [style=light] (25) to (22);
				\draw [style=light] (22) to (26);
				\draw [style=light] (26) to (23);
				\draw [style=light] (20) to (24);
				\draw [style=light] (24) to (21);
				\draw [style=light] (20) to (25);
				\draw [style=light] (25) to (21);
				\draw [style=light] (20) to (26);
				\draw [style=light] (26) to (21);
				\draw [style=light] (20) to (27);
				\draw [style=light] (27) to (21);
				\draw [style=light] (27) to (23);
				\draw [style=light] (27) to (22);
				\draw (27) to (26);
				\draw (25) to (24);
				\draw [style=light] (35) to (31);
				\draw [style=light] (35) to (33);
				\draw [style=light] (35) to (28);
				\draw [style=light] (35) to (32);
				\draw [style=light, bend left=15] (35) to (30);
				\draw [style=light] (34) to (28);
				\draw [style=light] (35) to (29);
				\draw [style=light] (34) to (31);
				\draw [style=light] (34) to (30);
				\draw [style=light] (34) to (33);
				\draw [style=light, bend left=15] (34) to (29);
				\draw [style=light] (34) to (32);
				\draw [style=light] (31) to (36);
				\draw [style=light, bend right=15] (36) to (28);
				\draw [style=light] (32) to (36);
				\draw [style=light] (36) to (29);
				\draw [style=light] (36) to (30);
				\draw [style=light] (36) to (33);
				\draw (29) to (32);
				\draw (29) to (33);
				\draw (29) to (31);
				\draw (31) to (28);
				\draw (28) to (32);
				\draw (28) to (33);
				\draw (30) to (32);
				\draw (30) to (33);
				\draw (30) to (31);
			\end{pgfonlayer}
		\end{tikzpicture}
		
		\caption{The graphs $F(0)$ through $F(6)$. Where the graph is constructed by a join of the form $G+H$, the vertices of $G$ are black, vertices of $H$ are white, and the edges between $G$ and $H$ are lighter for clarity.}
		\label{fig:f_sharp_examples}
	\end{figure}
	
	Using Lemma \ref{lem:join_gamma_2}, and checking $F(3), F(4)$ and $F(6)$ by hand, we see that for all $\delta \geq 3$, $F(\delta)$ has power domination number 2, power domination radius 1, minimum degree $\delta$ and $\delta + 3$ vertices.
	Thus $F(\delta)$ attains the bound $rad_p(F(\delta)) = (\delta + 3) - \delta - 2$ in Theorem \ref{thm:upper_bound_geq_2}.
	A similar family of sharpness examples can be constructed by repeatedly taking joins of $F(0), F(1)$ and $F(5)$ with $\overline{K_3}$.

	\section{Split Graphs}
	\label{sec:split_graphs}
	
	In this section, we give a sharp upper bound on the power domination radius of a split graph in terms of its order and power domination number. A \textit{split graph} is a graph in which the vertex set can be partitioned into an independent set and a set which induces a clique.
	Throughout this section, $G = (C \sqcup I, E)$ will denote a connected split graph with vertex set $V = C \sqcup I$ such that $C$ and $I$ are non-empty, $G[C]$ is a maximum clique of $G$, and $I$ is an independent set.
	Despite their simple structure, split graphs can have a very large power domination radius of almost $\frac{n}{2}$.
	Further, the (rather large) bound $rad_p(G) \leq \omega(G) = |C|$ is best possible without further constraints.
	The situation is less dire when $\gamma_p$ is large.
	We begin by showing that a split graph has an optimal PDS contained in its clique $C$.
	
	\begin{rem}
		\label{rem:split_i_vertex}
		If $v$ is a vertex of $I$, and $u$ is any neighbour of $v$, then $u\in C$ and $N[v] \subseteq N[u]$.
	\end{rem}
	
	\begin{lem}
		\label{lem:optimal_s_clique}
		Let $G = (C\cup I, E)$ be a connected split graph with $C$ and $I$ non-empty.
		Then $G$ contains an optimal PDS, $S$, such that $S\subseteq N(I) \subseteq C$.
	\end{lem}
	
	\begin{proof}
		Assume to the contrary that $G$ does not contain such an optimal PDS. 
		Among all optimal PDS's of $G$, let $T$ be the one minimising the cardinality $|T \setminus N(I)| > 0$.
		As $T$ is not contained in $N(I)$, there exists a vertex $v$ of $T$ such that either $v\in I$, or $v\in C - N(I)$.
		
		If $v\in I$, then by Remark \ref{rem:split_i_vertex}, $v$ has a neighbour $u$ in $C$ such that $N(v) \subseteq N(u)$.
		The set $(T \setminus \{v\}) \cup \{u\}$ is an optimal PDS (since exchanging $v$ for its neighbour $u$ cannot increase the power domination radius), contradicting the minimality of $T$.
		
		If $v \in C \setminus N(I)$, there are two cases to consider. 
		If there is a vertex of $T\cap C$ other than $v$, then $T \setminus \{v\}$ is a PDS, contradicting the choice of $T$.
		If $v$ is the only vertex of $T\cap C$, then $v$ has a neighbour $w\in C \cap N(I)$.
		Thus $(T \setminus \{v\}) \cup \{w\}$ is an optimal PDS, contradicting the choice of $T$.
	\end{proof}
	
	\begin{prop}
		\label{prop:split_clique_bound}
		Let $G$ be a split graph.
		Then $rad_p(G) \leq \omega(G) - \gamma_p(G) + 1$.
	\end{prop}
	
	\begin{proof}
		Per Lemma \ref{lem:optimal_s_clique}, $G$ has an optimal PDS $S$ that is contained in $N(I)$.
		Note that $N(I) \subseteq C \subseteq N[S]$.
		Since the vertices adjacent to $I$ are already monitored in the first step, no vertex of $I$ will ever propagate in the propagation sequence derived from $S$.
		Thus the only vertices that propagate after the dominating step are those of $C \setminus S$. 
		So the propagation radius is at most $1 + |C \setminus S| = 1 + \omega(G) - \gamma_p(G)$.
	\end{proof}
	
	\begin{thm}
		\label{thm:split_graph}
		Suppose $G$ is a split graph of order $n$ and power domination number $\gamma_p$.
		Then
		\[
		rad_p(G) \leq \frac{n - 3\gamma_p}{2} + 1.
		\]
	\end{thm}
	
	\begin{proof}
		Let the vertex set of $G$ be $V = C\sqcup I$.
		Applying Lemma \ref{lem:optimal_s_clique}, we see that $G$ has an optimal PDS $S \subseteq N(I) \subseteq C$.
		In particular, the vertices of $S$ are all adjacent.
		Further, among all optimal PDS's contained in $N(I)$, choose $S$ such that $|N[S]|$ is maximum.
		Note that if each vertex of $C$ has at most one neighbour in $I$, then every vertex of $C$ with such a neighbour can propagate after the first step.
		Thus we have $rad_p(G) \leq 2$.
		Therefore we may assume without loss of generality that some vertex of $C$ has at least two neighbours in $I$.
		
		We further claim that each vertex of $S$ has at least two external private neighbours in $I$. 
		There are two cases to consider based on the cardinality $|S|$.
		If $S$ has exactly one vertex, then by maximality of $S$, that one vertex in $S$ must be one with at least two (necessarily external private) neighbours in $I$.
		Suppose now that $|S| \geq 2$.
		If $u\in S$ has at most one external private neighbour, then $S \setminus \{u\}$ is a PDS, contradicting that $S$ is optimal and proving the claim.
		By this claim, we have that $|N(S) \cap I| \geq 2|S| = 2\gamma_p$.
		
		Note that in each step of the propagation sequence after the first, a vertex of $C - S$ propagates to a vertex of $I \setminus N(S)$.
		There are $rad_p(G) - 1$ such steps.
		We now count the number of vertices $n$ of $G$.
		
		\begin{align*}
			n &\geq |S| + |N(S) \cap I| + |C\setminus S| + |I \setminus N(S)|\\
			n &\geq \gamma_p + 2\gamma_p + (rad_p(G) - 1) + (rad_p(G) - 1)
		\end{align*}
		
		Rearranging the last equation completes the proof.
	\end{proof}
	
	\begin{figure}[ht]
		\centering
		\begin{tikzpicture}
			\begin{pgfonlayer}{nodelayer}
				\node [style=blackvertex] (0) at (-3, 0) {};
				\node [style=blackvertex] (1) at (-2, 1) {};
				\node [style=whitevertex] (2) at (-0.5, 1.5) {};
				\node [style=whitevertex] (3) at (1, 1) {};
				\node [style=whitevertex] (4) at (2, 0) {};
				\node [style=whitevertex, label={above:$x_1$}] (5) at (-4.5, 3) {};
				\node [style=whitevertex, label={above:$y_1$}] (6) at (-3.5, 3) {};
				\node [style=whitevertex, label={above:$x_2$}] (7) at (-2.5, 3) {};
				\node [style=whitevertex, label={above:$y_2$}] (8) at (-1.5, 3) {};
				\node [style=whitevertex, label={above:$z_1$}] (9) at (0, 3) {};
				\node [style=whitevertex, label={above:$z_2$}] (10) at (1.5, 3) {};
				\node [style=whitevertex, label={above:$z_3$}] (11) at (3, 3) {};
			\end{pgfonlayer}
			\begin{pgfonlayer}{edgelayer}
				\draw (0) to (2);
				\draw (1) to (2);
				\draw (1) to (0);
				\draw (0) to (4);
				\draw (3) to (4);
				\draw (2) to (3);
				\draw (1) to (3);
				\draw (2) to (4);
				\draw (0) to (3);
				\draw (1) to (4);
				\draw (5) to (0);
				\draw (6) to (0);
				\draw (7) to (1);
				\draw (1) to (8);
				\draw (2) to (9);
				\draw (9) to (3);
				\draw (3) to (10);
				\draw (10) to (4);
				\draw (4) to (11);
			\end{pgfonlayer}
		\end{tikzpicture}
		
		\caption{The graph $S(4,2)$ shown is a split graph with power domination number 2 and power domination radius 4. The vertices $s_1$ and $s_2$ of an optimal PDS are bolded.}
		\label{fig:split_graph_sharp}
	\end{figure}
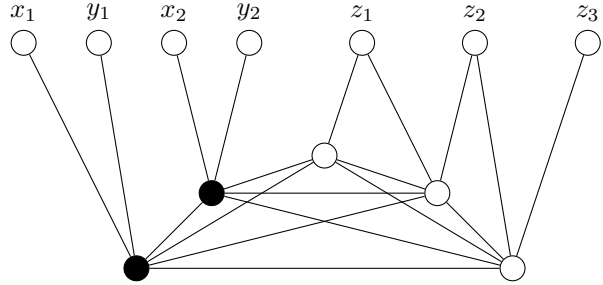
	
	The bound in Theorem \ref{thm:split_graph} is sharp for all combinations of $rad_p$ and $\gamma_p$.
	For positive integers $p$ and $g$, we construct a split graph $S(p, g) = (V,E)$ as follows.
	Let $V$ be the disjoint union of five sets $S, T, X, Y, Z$ with $S = \{s_1, \dots, s_g\}$, $X = \{x_1, \dots, x_g\}$, $Y = \{y_1, \dots, y_g\}$, $T = \{t_1, \dots, t_{p-1}\}$ and $Z = \{z_1, \dots, z_{p-1}\}$.
	The edge set $E$ contains: all edges of the form $uv$ where both $u$ and $v$ belong to $S\cup T$; all edges $s_ix_i$; all edges $s_iy_i$; all edges $t_iz_i$ and all edges $t_{i-1}z_i$ where $i\geq 2$.
	The resulting graph $S(p,g)$ is a split graph with clique $S\cup T$ and independent set $X\cup Y \cup Z$ (See Figure \ref{fig:split_graph_sharp}).
	It has power domination radius $p$, power domination number $g$ and order $n = 3g + 2p - 2$, so it attains the bound in Theorem \ref{thm:split_graph}.

	\section{The sharp lower bound}
	\label{sec:low_bound}
	
	In \cite{pdbtc_liao_2016}, Liao determines a lower bound on the power domination number in terms of the power domination radius:
	
	\begin{thm}\textup{\cite{pdbtc_liao_2016}}
		Let $G$ be a connected graph with order $n$, maximum degree $\Delta$ and power domination radius $rad_p$. Then
		\[
		\gamma_p(G) \geq \frac{n}{rad_p\cdot \Delta + 1}
		\]
		\label{thm:lower_bound_g}
	\end{thm}
	
	This can be easily re-arranged to give a lower bound for the power domination radius:
	
	\begin{cor}
		Let $G$ be a connected graph of order $n$, maximum degree $\Delta$ and power domination number $\gamma_p$. Then
		\[
		rad_p(G) \geq \frac{n-\gamma_p}{\gamma_p\cdot \Delta}
		\]
		\label{cor:lower_bound_p}
	\end{cor}
	
	In \cite{pdbtc_liao_2016}, Theorem \ref{thm:lower_bound_g} and Corollary \ref{cor:lower_bound_p} are shown to be sharp for all possible values of $rad_p$ and $\gamma_p$ respectively.
	However, $\Delta = 4$ in all of the extremal examples. 
	We show that for all $\Delta \geq 3$, $rad_p\geq 1$ and $\gamma_p \geq 1$, there is a graph $H(\Delta, \gamma_p, rad_p)$ with maximum degree $\Delta$, power domination number $\gamma_p$, power domination radius $rad_p$ and order $n = \gamma_p(rad_p\cdot \Delta + 1)$.
	
	To construct $H(\Delta, \gamma_p, rad_p)$, begin with $\gamma_p$ disjoint copies of the star $K_{1,\Delta}$. 
	Replace every edge with a path of length $rad_p$.
	For $i\in [1, \gamma_p]$, let $u_i$ denote the vertex of degree $\Delta$ in the $i^{th}$ subdivided star, and let $v_i$ denote some leaf of the $i^{th}$ subdivided star.
	Add all possible edges of the form $v_iv_{i+1}$ to complete the construction (see Figure \ref{fig:h_5_4_2}).
	
	\begin{figure}[ht]
		\centering
		\begin{tikzpicture}[scale=0.6]
			\begin{pgfonlayer}{nodelayer}
				\node [style=whitevertex, label={below:$v_2$}] (0) at (-2, 0) {};
				\node [style=blackvertex, label={above:$u_2$}] (1) at (-2, 3.5) {};
				\node [style=whitevertex] (2) at (-3.25, 4.5) {};
				\node [style=whitevertex] (3) at (-3.25, 5.75) {};
				\node [style=whitevertex] (4) at (-0.75, 4.5) {};
				\node [style=whitevertex] (5) at (-0.75, 5.75) {};
				\node [style=whitevertex] (7) at (-2, 1.75) {};
				\node [style=whitevertex] (8) at (-0.75, 2.5) {};
				\node [style=whitevertex] (9) at (-0.75, 1) {};
				\node [style=whitevertex] (10) at (-3.25, 2.5) {};
				\node [style=whitevertex] (11) at (-3.25, 1) {};
				\node [style=whitevertex] (12) at (-3.25, 1) {};
				\node [style=whitevertex, label={below:$v_1$}] (13) at (-6, 0) {};
				\node [style=blackvertex, label={above:$u_1$}] (14) at (-6, 3.5) {};
				\node [style=whitevertex] (15) at (-7.25, 4.5) {};
				\node [style=whitevertex] (16) at (-7.25, 5.75) {};
				\node [style=whitevertex] (17) at (-4.75, 4.5) {};
				\node [style=whitevertex] (18) at (-4.75, 5.75) {};
				\node [style=whitevertex] (19) at (-6, 1.75) {};
				\node [style=whitevertex] (20) at (-4.75, 2.5) {};
				\node [style=whitevertex] (21) at (-4.75, 1) {};
				\node [style=whitevertex] (22) at (-7.25, 2.5) {};
				\node [style=whitevertex] (23) at (-7.25, 1) {};
				\node [style=whitevertex] (24) at (-7.25, 1) {};
				\node [style=whitevertex, label={below:$v_3$}] (25) at (2, 0) {};
				\node [style=blackvertex, label={above:$u_3$}] (26) at (2, 3.5) {};
				\node [style=whitevertex] (27) at (0.75, 4.5) {};
				\node [style=whitevertex] (28) at (0.75, 5.75) {};
				\node [style=whitevertex] (29) at (3.25, 4.5) {};
				\node [style=whitevertex] (30) at (3.25, 5.75) {};
				\node [style=whitevertex] (31) at (2, 1.75) {};
				\node [style=whitevertex] (32) at (3.25, 2.5) {};
				\node [style=whitevertex] (33) at (3.25, 1) {};
				\node [style=whitevertex] (34) at (0.75, 2.5) {};
				\node [style=whitevertex] (35) at (0.75, 1) {};
				\node [style=whitevertex] (36) at (0.75, 1) {};
				\node [style=whitevertex, label={below:$v_4$}] (37) at (6, 0) {};
				\node [style=blackvertex, label={above:$u_4$}] (38) at (6, 3.5) {};
				\node [style=whitevertex] (39) at (4.75, 4.5) {};
				\node [style=whitevertex] (40) at (4.75, 5.75) {};
				\node [style=whitevertex] (41) at (7.25, 4.5) {};
				\node [style=whitevertex] (42) at (7.25, 5.75) {};
				\node [style=whitevertex] (43) at (6, 1.75) {};
				\node [style=whitevertex] (44) at (7.25, 2.5) {};
				\node [style=whitevertex] (45) at (7.25, 1) {};
				\node [style=whitevertex] (46) at (4.75, 2.5) {};
				\node [style=whitevertex] (47) at (4.75, 1) {};
				\node [style=whitevertex] (48) at (4.75, 1) {};
			\end{pgfonlayer}
			\begin{pgfonlayer}{edgelayer}
				\draw (3) to (2);
				\draw (2) to (1);
				\draw (4) to (1);
				\draw (4) to (5);
				\draw (1) to (10);
				\draw (10) to (12);
				\draw (1) to (7);
				\draw (7) to (0);
				\draw (1) to (8);
				\draw (8) to (9);
				\draw (16) to (15);
				\draw (15) to (14);
				\draw (17) to (14);
				\draw (17) to (18);
				\draw (14) to (22);
				\draw (22) to (24);
				\draw (14) to (19);
				\draw (19) to (13);
				\draw (14) to (20);
				\draw (20) to (21);
				\draw (28) to (27);
				\draw (27) to (26);
				\draw (29) to (26);
				\draw (29) to (30);
				\draw (26) to (34);
				\draw (34) to (36);
				\draw (26) to (31);
				\draw (31) to (25);
				\draw (26) to (32);
				\draw (32) to (33);
				\draw (40) to (39);
				\draw (39) to (38);
				\draw (41) to (38);
				\draw (41) to (42);
				\draw (38) to (46);
				\draw (46) to (48);
				\draw (38) to (43);
				\draw (43) to (37);
				\draw (38) to (44);
				\draw (44) to (45);
				\draw (13) to (0);
				\draw (0) to (25);
				\draw (25) to (37);
			\end{pgfonlayer}
		\end{tikzpicture}
		
		\caption{The graph $H(5,4,2)$ has order $n = 44 = \gamma_p(rad_p\cdot\Delta + 1).$}
		\label{fig:h_5_4_2}
	\end{figure}
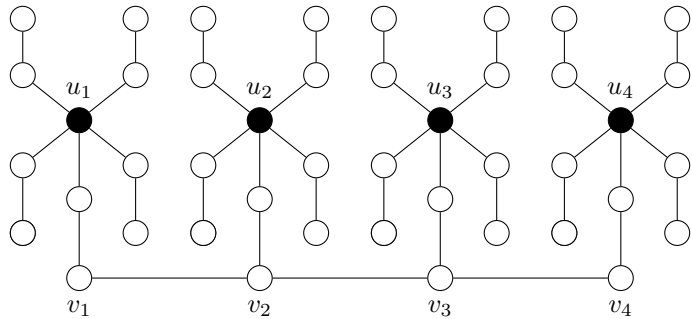
	
	It's clear that the set $\{u_1, u_2, \dots, u_{\gamma_p}\}$ is an optimal power dominating set with $\gamma_p$ vertices and radius $rad_p$.
	
	\section{Conclusion and further questions}
	
	In section \ref{sec:upp_gam_1}, we presented a sharp upper bound on the power domination radius of a graph, $rad_p$, with given order $n$ and minimum degree $\delta$. 
	This bound is sharp for all $\delta \geq 2$, but only for graphs with $\gamma_p = 1$.
	For all $\delta \geq 2$, the bound can be attained by a regular graph (section \ref{sec:sharp_reg}). 
	In section \ref{sec:upp_gam_2}, we give a bound on the power domination radius in terms of $n$, $\delta$ and $\gamma_p$. 
	The bound is sharp whenever either $\gamma_p = 2$ or $\delta = 2$, however we do not know if this bound is best possible in other cases.
	Section \ref{sec:split_graphs} presents a sharp upper bound on $rad_p$ for split graphs in terms of $\gamma_p$ and $n$.
	We conclude in section \ref{sec:low_bound} by showing that a result of Liao \cite{pdbtc_liao_2016} yields a sharp lower bound on $rad_p$ in terms of $n$, $\Delta$ and $\gamma_p$.
	This leaves the following questions: 
	\begin{enumerate}
		\item Can the bound in Theorem \ref{thm:upper_bound_geq_2} be improved when $\delta \geq 3$ and $\gamma_p \geq 3$?
		If not, can we construct a family of graphs attaining the bound for all combinations of $\delta \geq 3$ and $\gamma_p \geq 3$?
		
		\item Can the upper and lower bounds on $rad_p$ in Theorems \ref{thm:delta_1_up_bound}, \ref{thm:upper_bound_geq_2} and \ref{thm:lower_bound_g} be improved for interesting classes of graphs?
	\end{enumerate}

	\section*{Statements and Declarations}
	
	This research is supported in part by the DSI-NRF Centre of Excellence in Mathematical and Statistical Sciences (CoE-MaSS), South Africa. Opinions expressed and conclusions arrived at are those of the author and are not necessarily to be attributed to the CoE-MaSS.
	

\end{document}